\documentclass{amsart}

\usepackage{ae}
\usepackage[all]{xy}
\usepackage{graphicx,amsfonts,amssymb,amsmath}
\usepackage{natbib}

\usepackage[latin1]{inputenc}
\usepackage[T1]{fontenc}
\usepackage[english]{babel}
\usepackage[cyr]{aeguill}

\newcommand{\chapeau}{{\rlap{\smash{\hbox{\lower4pt\hbox{\hskip1pt$\widehat{\phantom{u}}$}}}}}\mbox{ }}
\usepackage[OT2,T1]{fontenc}
\DeclareSymbolFont{cyrletters}{OT2}{wncyr}{m}{n}
\DeclareMathSymbol{\sha}{\mathalpha}{cyrletters}{"58}
\input cyracc.def
\font\tencyr=wncyr10
\def\cyr{\tencyr\cyracc}
\newcommand{\ba}{\mbox{\cyr B}}

 \newtheorem{thm}{Theorem}[subsection]
 \newtheorem*{thmA}{Theorem A}
 \newtheorem*{thmB}{Theorem B}

 \newtheorem*{pb*}{\textit{Question}}
 \newtheorem*{cor*}{\textit{Corollary}}

 \newtheorem{cor}[thm]{Corollary}
 \newtheorem{lem}[thm]{Lemma}
 \newtheorem{prop}[thm]{Proposition}
 \theoremstyle{definition}
 \newtheorem{defn}[thm]{Definition}
 \theoremstyle{remark}
 
 \theoremstyle{remark}

 \newtheorem{rem}[thm]{Remark}
 \numberwithin{equation}{subsection}

 \newcommand{\To}{\longrightarrow}

\usepackage[dvipdfm,colorlinks]{hyperref}
\hypersetup{colorlinks,linkcolor=blue,citecolor=blue,}

\begin{document}

\title[Arithmetic of 0-cycles]
{Arithmetic of 0-cycles on varieties defined over number fields}

\author{ Yongqi LIANG  }

\address{Yongqi LIANG \newline
Département de Mathématiques, \newline Bâtiment 425,\newline Université  Paris-sud 11,\newline  F-91405 Orsay,\newline
 France}

\email{yongqi.liang@math.u-psud.fr}

\thanks{\textit{Key words} : zero-cycles , Hasse principle, weak approximation,
Brauer-Manin obstruction, rationally connected varieties, homogeneous spaces}

\thanks{\textit{MSC 2010} : 14G25 (11G35, 14M22)}

\date{\today.}



\maketitle

\small
\textsc{Abstract.}
Let $X$ be a rationally connected algebraic variety, defined over a number field $k.$
We find a relation between the arithmetic of rational points on $X$ and the arithmetic
of zero-cycles. More precisely, we consider the following statements:
(1) the Brauer-Manin obstruction is the only obstruction to weak approximation for $K$-rational points on
$X_K$ for all finite extensions $K/k;$
(2) the Brauer-Manin obstruction is the only obstruction to weak approximation in some sense that we define
for zero-cycles of degree $1$ on $X_K$ for all finite extensions $K/k;$
(3) 
the sequence $$\varprojlim_n CH_0(X_K)/n\to\prod_{w\in\Omega_K}\varprojlim_nCH_0'(X_{K_w})/n\to Hom(Br(X_K),\mathbb{Q}/\mathbb{Z})$$
is exact for all finite extensions $K/k.$ We prove that (1) implies (2), and that (2) and (3) are equivalent.
We also prove a similar implication for the Hasse principle.

As an application, we prove the exactness of the sequence above for smooth compactifications of
certain homogeneous spaces of linear algebraic groups.

\bigskip

\scriptsize
\begin{center}\textbf{Arithmétique des 0-cycles pour certaines vari\'et\'es d\'efinies sur les corps de nombres}\end{center}

\textsc{Résumé.}
Soit $X$ une variété algébrique rationnellement connexe, définie sur un corps de nombres $k.$
On trouve, sur $X,$ un lien entre l'arithmétique des points rationnels et l'arithmétique des
zéro-cycles. Plus précisément, on considère les assertions suivantes:
(1) l'obstruction de Brauer-Manin est la seule à l'approximation faible pour les points $K$-rationnels sur
$X_K$ pour toute extension finie $K/k;$
(2) l'obstruction de Brauer-Manin est la seule à l'approximation faible (en un certain sens à préciser) pour les
zéro-cycles de degré $1$ sur $X_K$ pour toute extension finie $K/k;$
(3) 
la suite $$\varprojlim_n CH_0(X_K)/n\to\prod_{w\in\Omega_K}\varprojlim_nCH_0'(X_{K_w})/n\to Hom(Br(X_K),\mathbb{Q}/\mathbb{Z})$$
est exacte pour toute extension finie $K/k.$ On démontre que (1) implique (2), et que (2) et (3) sont équivalentes.
On trouve également une implication similaire pour le principe de Hasse.

Comme application, on montre l'exactitude de la suite ci-dessus pour les compactifications lisses
de certains espaces homogènes de groupes algébriques linéaires.

\normalsize

\tableofcontents

\section*{Introduction}

\subsection*{Brauer-Manin obstruction, for rational points and for 0-cycles}

Let $X$ be a proper smooth algebraic variety defined over a number field $k.$
We denote by $\Omega_k$ the set of all the places of $k,$ and by $k_v$ the completion of $k$ with respect to $v\in\Omega_k.$
Let $Br(X)=H^2_{\mbox{\scriptsize\'et}}(X,\mathbb{G}_m)$ be the cohomological Brauer group of $X.$
In \cite{Manin}, Manin defined the following pairing (called the \emph{Brauer-Manin pairing}):
\begin{equation*}
    \begin{array}{rcccl}
     \langle\cdot,\cdot \rangle _k:\prod_{v\in\Omega_k}X(k_v) & \times & Br(X) & \to & \mathbb{Q}/\mathbb{Z},\\
         (\mbox{ }\{x_v\}_{v\in\Omega_k} & , & b\mbox{ }) & \mapsto & \sum_{v\in\Omega_k}inv_v(b(x_v)),\\
    \end{array}
\end{equation*}
where $inv_v:Br(k_v)\hookrightarrow\mathbb{Q}/\mathbb{Z}$ is the local invariant at $v,$
and where $b(x_v)$ is the evaluation of $b$ at the point $x_v,$ \emph{i.e.} the pull back of $b\in Br(X)$ via the morphism
$x_v:Spec(k_v)\to X.$ According to the exact sequence $$0\to Br(k)\to\bigoplus_vBr(k_v)\to\mathbb{Q}/\mathbb{Z}\to0$$
coming from class field theory, the pairing factorizes through $Br(X)/Br(k)$ and
the left kernel $\left[\prod_vX(k_v)\right]^{Br}$ of the pairing contains
$X(k),$ and its closure $\overline{X(k)}$ in $\prod_vX(k_v)$ by continuity.
The condition $\left[\prod_vX(k_v)\right]^{Br}=\emptyset$
gives an obstruction to the existence of a global rational point.
This explained most known examples for the failure of Hasse principle
at that time.
If $\left[\prod_vX(k_v)\right]^{Br}\neq\emptyset$ implies that
$X$ has a $k$-rational point, we say that \emph{the Brauer-Manin obstruction is the only obstruction to the Hasse principle
for rational points}. Similarly, the condition $\left[\prod_vX(k_v)\right]^{Br}\subsetneq\prod_vX(k_v)$ gives
an obstruction to the weak approximation property; if $\overline{X(k)}=\left[\prod_vX(k_v)\right]^{Br}$
we say that \emph{the Brauer-Manin obstruction is
the only obstruction to weak approximation for rational points}.

\begin{pb*}
On which families of algebraic varieties, is
the Brauer-Manin obstruction the only obstruction to the Hasse principle/weak approximation for rational points?
\end{pb*}

The question dates back to a paper of Colliot-Thélène and Sansuc
\cite{CTSansuc77-3}, see the Bourbaki expos\'e by Peyre \cite{Peyre} for a recent survey of the problem.

\bigskip
Let us write $X_K=X\otimes_kK$ for any field extension $K$ of $k.$
In \cite{CT95}, Colliot-Th\'el\`ene extended the Brauer-Manin pairing to 0-cycles:
\begin{equation*}
    \begin{array}{rcccl}
     \langle\cdot,\cdot \rangle _k:\prod_{v\in\Omega_k}Z_0(X_{k_v}) & \times & Br(X) & \to & \mathbb{Q}/\mathbb{Z},\\
         (\mbox{ }\{z_v\}_{v\in\Omega_k} & , & b\mbox{ }) & \mapsto & \sum_{v\in\Omega_k}inv_v(\langle z_v,b \rangle _{k_v}),\\
    \end{array}
\end{equation*}
where $\langle\cdot,\cdot\rangle_{k_v}$ is defined as follows,
\begin{equation*}
    \begin{array}{rcccl}
     \langle\cdot,\cdot\rangle_{k_v}:   Z_0(X_{k_v}) & \times & Br(X) & \to & Br(k_v),\\
          (\mbox{ }\sum_Pn_PP &,& b\mbox{ }) & \mapsto & \sum_Pn_Pcores_{k_v(P)/k_v}(b(P)).\\
    \end{array}
\end{equation*}
He conjectured that the Brauer-Manin obstruction is the only obstruction to the Hasse principle
for 0-cycles of degree $1$ on any smooth proper variety. This conjecture is proved for any curve (supposing finiteness
of the Tate-Shafarevich group of its Jacobian) by Saito \cite{Saito}.

As a variant of \cite[page 69]{CT-SD}, we say that \emph{the Brauer-Manin obstruction is
the only obstruction to weak approximation for 0-cycles of degree $1$}, if for any positive integer $n$ and
for any finite subset $S\subset\Omega_k,$ given an arbitrary family $\{z_v\}\bot Br(X)$ of local 0-cycles of degree $1,$
then there exists a global 0-cycle $z_{n,S}$ of degree $1$ such that $z_{n,S}$ and $z_v$ have the same image in $CH_0(X_{k_v})/n$
for any $v\in S,$ \emph{cf.} \S \ref{WA 0-cyc}.
This notion turns out to be closely related to the exactness of the sequences $(E)$ and $(E_0)$ described as follows.

\subsection*{Exact sequences of local-global type}

We define the modified Chow group $CH'_0(X_{k_v})$ to be the usual Chow group $CH_0(X_{k_v})$ of 0-cycles for each
non-archimedean place $v,$ otherwise we set $CH'_0(X_{k_v})=CH_0(X_{k_v})/N_{\bar{k}_v|k_v}CH_0(X_{\bar{k}_v})$ for archimedean
places where $N_{\bar{k}_v|k_v}$ represents the norm.
The above pairing for 0-cycles factorizes through the modified Chow groups
$CH_0'(X_{k_v}).$
Let us denote $M^\chapeau=\varprojlim_nM/nM$ for any abelian group $M.$
The exactness of the complex induced by the Brauer-Manin pairing
$$CH_0(X)^\chapeau\to\prod_{v\in\Omega}CH'_0(X_{k_v})^\chapeau\to Hom(Br(X),\mathbb{Q}/\mathbb{Z}) \leqno(E)$$
is conjectured by Colliot-Th\'el\`ene/Sansuc \cite{CTSansuc81} for rational surfaces,
by Kato/Saito \cite{KatoSaito86} and Colliot-Th\'el\`ene \cite{CT95} for all smooth proper varieties;
it has been reformulated into this form by van Hamel \cite{vanHamel} and Wittenberg \cite{Wittenberg}.

Let $A_0(X)=Ker[deg:CH_0(X)\to\mathbb{Z}].$ We can consider the degree $0$ part of the sequence $(E):$
$$A_0(X)^\chapeau\to\prod_{v\in\Omega}A_0(X_{k_v})^\chapeau\to Hom(Br(X),\mathbb{Q}/\mathbb{Z})\leqno(E_0)$$
If $X$ is a smooth projective curve such that the Tate-Shafarevich group $\sha(Jac(X),k)$ of its Jacobian is finite,
then Colliot-Th\'el\`ene proved
the exactness of $(E)$ for $X$ in \cite[Proposition 3.3]{CT99HP0-cyc}.
In particular, he reproved (part of) the Cassels-Tate exact sequence when $X$ is an elliptic curve.

\subsection*{Rational points versus 0-cycles}

One may ask whether there are some relations between the arithmetic of rational points and
the arithmetic of 0-cycles.
Some recent examples (\emph{cf.} \S \ref{remark}) show that one should not expect an affirmative answer
for general proper smooth varieties.
But sometimes these two questions are discussed in parallel using similar methods, for example
\cite{CT-SD} \cite{CT-Sk-SD}.

\subsection*{Main results} In the present work, mainly for rationally connected varieties,
we find out a relation between the Brauer-Manin obstructions for rational points
and for 0-cycles on rationally connected varieties.
As an application, we prove the exactness of $(E)$ and $(E_0)$ for certain homogeneous spaces of linear
algebraic groups. To be precise, we state our main results as follows.

\bigskip
\emph{Let $X$ be a proper smooth variety defined over a number field $k.$
Consider the following statements:}
\begin{enumerate}
\item[(pt-HP)] \emph{the Brauer-Manin obstruction is the only obstruction to the Hasse principle for $K$-rational points
on $X_K$ for any finite extension $K/k;$}

\item[(pt-WA)] \emph{the Brauer-Manin obstruction is the only obstruction to weak approximation for $K$-rational points
on $X_K$ for any finite extension $K/k;$}

\item[(0cyc-HP)] \emph{the Brauer-Manin obstruction is the only obstruction to the Hasse principle for 0-cycles of degree $1$
on $X_K$ for any finite extension $K/k;$}

\item[(0cyc-WA)] \emph{the Brauer-Manin obstruction is the only obstruction to weak approximation for 0-cycles of degree $1$
on $X_K$ for any finite extension $K/k;$}

\item[(\textsc{Exact})] \emph{the sequence $(E)$ is exact for $X_K,$ and therefore so is $(E_0),$ for any finite extension $K/k.$}
\end{enumerate}

\begin{thmA}
The statement \emph{(\textsc{Exact})} implies \emph{(0cyc-HP)} and \emph{(0cyc-WA)}.
If moreover $X$ is \emph{rationally connected}, then
\emph{(0cyc-WA)} is equivalent to \emph{(\textsc{Exact})}.
\end{thmA}

\begin{thmB}[Theorem \ref{prop i-ii}]
If $X$ is a variety such that the N\'eron-Severi group $NS(X_{\bar{k}})$ is torsion-free
and $H^1(X_{\bar{k}},\mathcal{O}_{X_{\bar{k}}})=H^2(X_{\bar{k}},\mathcal{O}_{X_{\bar{k}}})=0,$
in particular if $X$ is \emph{rationally connected},
then \emph{(pt-HP)} implies \emph{(0cyc-HP)}, and \emph{(pt-WA)} implies \emph{(0cyc-WA)}.
\end{thmB}

Since the statements (pt-HP) and (pt-WA) are proved by Borovoi for smooth compactifications of certain
homogeneous spaces of linear algebraic groups \cite{Borovoi96}, we obtain the following

\begin{cor*}
Let $Y$ be a
homogeneous space of a connected linear algebraic group $G$ defined over a number field,
and let $X$ be a smooth compactification of $Y.$ Suppose that the geometric stabilizer
of $Y$ is connected (or abelian if $G$ is supposed semisimple simply connected).

Then the sequences $(E)$ and $(E_0)$ are exact for $X,$ and
the Brauer-Manin obstruction is the only obstruction to the Hasse principle/weak approximation for 0-cycles of degree $1$
on $X.$
\end{cor*}

The exactness of $(E)$ was not known before the present article even for
(smooth compactifications of) algebraic tori of arbitrary dimension.

\subsection*{Organization of the article}
After some preliminaries in \S \ref{notation}, we discuss in \S \ref{E vs WA}
the relation between the exactness of the sequence $(E)$
and weak approximation for 0-cycles, and prove the first assertion of Theorem A. In
\S \ref{BM obstruction} assuming rational connectedness we give a detailed proof of the second assertion
of Theorem A as well as Theorem B, which divides into
several steps with more precise statements.
Finally in \S \ref{remark}, we make some remarks on Theorems A, B, and the corollary.

\bigskip
\emph{Acknowledgements.}
The author thanks his thesis advisor D. Harari for fruitful discussions and J.-L. Colliot-Th\'el\`ene
for his suggestions.
Thanks to O. Wittenberg for his contribution to this subject, by which the present work is enlightened.

\section{Notation and preliminaries}\label{notation}

In the present article, the base field will always be a number field $k,$
we denote by $\Omega_k$ (respectively $\Omega_k^\textmd{f},$ $\Omega_k^\infty,$ $\Omega_k^\mathbb{R},$ $\Omega_k^\mathbb{C}$)
the set of all the places (respectively finite places, archimedean places, real places, complex places) of $k.$
For a finite set $S\supset\Omega_k^\infty$ of places of $k,$ we write $O_{k,S}$ for the ring of $S$-integers.
We denote by $k_v,$ by $O_v,$ and by $k(v)$ the completion of $k$ with respect to each $v\in\Omega_k,$ its ring of integers, and its
residual field if $v$ is non-archimedean.
We fix for $k$ and for each $k_v$ an algebraic closure $\bar{k}$ and $\bar{k}_v$ such that the following diagram is commutative.
\SelectTips{eu}{12}$$\xymatrix@C=20pt @R=14pt{
\bar{k}\ar[r]&\bar{k}_v
\\ k\ar[r]\ar[u]&k_v\ar[u]
}$$
Let $K$ be a finite extension of $k,$ we denote by $S\otimes_kK$ the set of the places of $K$ over places in $S\subset\Omega_k.$

We denote by $X$ a proper smooth algebraic variety (separated scheme of finite type, geometrically integral if not
otherwise stated) defined over $k.$
We write $X_K=X\otimes_kK$ for any field extension $K$ of $k,$ and
we write simply $X_v=X_{k_v},$ $\overline{X}=X_{\bar{k}},$ and $\overline{X}_v=X_{\bar{k}_v}.$

Let $K/k$ be a Galois extension, we denote its Galois group by $\Gamma_{K|k}$ and we denote
the Galois cohomology $H^i(\Gamma_{K|k},M)$ by $H^i(K|k,M)$ for any $\Gamma_{K|k}$-module $M;$
if moreover $K=\bar{k},$ we write simply $\Gamma_k=\Gamma_{\bar{k}|k}$ and $H^i(k,M)=H^i(K|k,M).$

\subsection{0-cycles}\label{subsec-0-cyc}
Let $z=\sum n_iP_i$ be a 0-cycle of $X$ (expressed as multiplicated sum of distinct closed points $P_i$ of $X$),
we say that it is \emph{separable} if $n_i\in\{0,1,-1\}.$

Given a closed point $P$ of $X_v,$ we fix a $k_v$-embedding $k_v(P)\to\bar{k}_v,$ the point $P$ is viewed
as a $k_v(P)$-rational point of $X_v.$ We say that a closed point $Q$ of $X_v$ is sufficiently close
to $P$ (with respect to a chosen neighborhood $U_P$ of $P$ in the topological space $X_v(k_v(P))$) if
$Q$ has residue field $k_v(Q)=k_v(P)$ and if we can choose a $k_v$-embedding $k_v(Q)\to\bar{k}_v$
such that $Q,$ viewed as a $k_v(Q)$-rational point of $X_v,$ is contained in $U_P.$
Extending $\mathbb{Z}$-linearly, for two 0-cycles $z_v$ and $z'_v$ on $X_v$ of the same degree,
we can say whether $z'_v$ is sufficiently close to $z_v.$
Let $b\in Br(X_v).$ Recall that if $z'_v$ is sufficiently close to $z_v,$ then
$\langle z'_v,b\rangle_{k_v}=\langle z_v,b\rangle_{k_v}$
according to the continuity of the evaluation of elements in the Brauer group, see \cite[Lemma 6.2]{B-Demarche}.
If $B$ is a finite subset of $Br(X_v),$ we say that $z'_v$ is close to $z_v$
with respect to $B$ if $\langle z'_v,b\rangle_{k_v}=\langle z_v,b\rangle_{k_v}$ for every $b\in B.$
Let $n$ be a positive integer. If $z'_v$ is sufficiently close to $z_v,$ then
these 0-cycles have the same image in $CH_0(X_v)/n,$ see \cite[Lemma 1.8]{Wittenberg}.
We say that $z'_v$ is close to $z_v$ with respect to $n,$ if the images of $z_v$ and $z'_v$
in $CH_0(X_v)/n$ coincide.

\subsection{Weak approximation for 0-cycles}\label{WA 0-cyc}

We define the notion of weak approximation for 0-cycles which is a variant of \cite[page 69]{CT-SD}.
Let $\delta$ be an integer. We say that the variety $X$ satisfies \emph{weak approximation for 0-cycles
of degree $\delta$}, if for any positive integer $n$ and for any finite subset $S\subset\Omega_k,$
given an arbitrary family of local 0-cycles $\{z_v\}_{v\in\Omega_k}\in\prod_{v\in\Omega_k}Z_0(X_v)$ each of degree $\delta,$
there exists a global 0-cycle $z=z_{n,S}$ of degree $\delta$ such that $z$ and $z_v$ have the same
image in $CH_0(X_v)/n$ for any $v\in S.$

If we only know that the last assertion is satisfied for those families $\{z_v\}$ such that $\{z_v\}\bot Br(X),$
then we say that \emph{the Brauer-Manin obstruction is the only obstruction to weak approximation for 0-cycles of
degree $\delta$}. If this holds for $X_K$ for every finite extension $K/k,$ we denote it by
$\mbox{(0cyc-WA)}^\delta,$ in particular, $\mbox{(0cyc-WA)}^1$ is exactly the statement $\mbox{(0cyc-WA)}$ in our main results.
By definition, the exactness of $(E_0)$ for $X$ implies the assertion that
the Brauer-Manin obstruction is the only obstruction to weak approximation for 0-cycles of
degree $0$ on $X.$

\subsection{Moving lemma for 0-cycles}

The following lemma is well-known, it will be used in the proof of the main results.

\begin{lem}\label{moving lemma}
Let $\pi:X\to\mathbb{P}^1$ be a dominant morphism between algebraic varieties defined over
$\mathbb{R},$ $\mathbb{C}$ or any finite extension of $\mathbb{Q}_p.$
Suppose that $X$ is smooth.

Then for any 0-cycle $z$ on $X,$ there exists a 0-cycle $z'$ on $X$
such that $\pi_*(z')$ is separable and
such that $z'$ is sufficiently close to $z.$
\end{lem}

\begin{proof}
Essentially, this result follows from the implicit function theorem. We find detailed arguments
in \cite[p.19]{CT-Sk-SD} and \cite[p.89]{CT-SD}.
\end{proof}

\subsection{Rationally connected varieties}

The following definition is one of various equivalent definitions due to Koll\'ar, Miyaoka and
Mori \cite[IV.3.]{Kollar}.

\begin{defn}
A proper smooth variety $X$ defined over an arbitrary field of characteristic zero $k$ is said to be
\emph{rationally connected},  if for some uncountable algebraically closed field $L$ containing $k,$
for any pair of points $P,Q\in X(L)$ there exists an $L$-morphism
$f:\mathbb{P}^1_L\to X_L$ such that $f(0)=P$ and $f(\infty)=Q.$
\end{defn}

The following property is also well-known for the family of rationally connected varieties.

\begin{prop}\label{finiteness of Br}
Let $X$ be a proper smooth variety defined over an algebraically closed field
of characteristic zero such that the N\'eron-Severi group $NS(X)$ is torsion-free and
$H^1(X,\mathcal{O}_X)=H^2(X,\mathcal{O}_X)=0,$ for example if $X$ is rationally connected.

Then the Picard group $Pic(X)$ is torsion-free and finitely generated, and the Brauer group
$Br(X)$ is finite.
\end{prop}

\begin{proof}
Different proofs can be found in the literature, we refer to \cite[II Corollaire 3.4]{Br}
and the proof of \cite[Proposition 2.11]{CTRaskind}, see also \cite[Corollary 4.18]{Debarre} and \cite[Lemma 1.4]{Liang2}.
\end{proof}

\section{Exactness of $(E)$ and weak approximation for 0-cycles}\label{E vs WA}

In this section, we assume only that the variety $X$ is proper and smooth.

\subsection{Real places in the sequence $(E)$}
In the sequence $(E)$ the simplest terms are the modified Chow groups for the non-archimedean places.
The group $CH'_0(X_v)$ is $0$ if $v$ is a complex place.
If $v$ is a real place of $k,$ the group $CH'_0(X_v)$
is calculated by Colliot-Th\'el\`ene/Ischebeck \cite[Proposition 3.2]{CTIschebeck}
for any proper smooth $\mathbb{R}$-varieties processing a 0-cycle of degree $1.$
We denote by $s$ the number of connected components (real topology)
of $X_v(\mathbb{R})$ (supposed non-empty), then $CH'_0(X_v)\simeq(\mathbb{Z}/2\mathbb{Z})^s.$

\subsection{$\mbox{(\textsc{Exact})}\Rightarrow\mbox{(0cyc-WA)}^\delta$}

Let $\delta$ be an integer, in this subsection we prove the implication
$\mbox{(\textsc{Exact})}\Rightarrow\mbox{(0cyc-WA)}^\delta$ assuming the existence of a global 0-cycle
of degree $1$ on $X.$

The following proposition is a more precise statement, with the remark immediately after the proof we see that
this completes the implication $\mbox{(\textsc{Exact})}\Rightarrow\mbox{(0cyc-HP)}$ and $\mbox{(0cyc-WA)}.$
It is a statement without base field extensions.
We will see later that the converse statement $\mbox{(0cyc-WA)}^\delta\Rightarrow\mbox{(\textsc{Exact})}$ is true for
rationally connected varieties, \S \ref{E for RC}.

\begin{prop}\label{E to ii}
Let $X$ be a proper smooth variety defined over a number field and let $\delta$ be an integer.
Assume the existence of a global 0-cycle of degree $1$ on $X.$
Then the exactness of $(E_0)$ for $X$ implies the assertion that the Brauer-Manin obstruction is the
only obstruction to weak approximation for 0-cycles of degree $\delta$ on $X.$
\end{prop}

\begin{proof}
Let $n$ be a positive integer and let $S$ be a finite set of places of $k.$ Suppose we are
given $\{z_v\}_{v\in\Omega_k}$ an arbitrary family of local 0-cycles of degree $\delta$ such that
$\{z_v\}\bot Br(X).$ We fix $z_0$ a global 0-cycle of degree $1,$ and get a family of local
0-cycles of degree zero $\{\delta z_0-z_v\}\bot Br(X).$ The exactness of $(E_0)$ gives us
a global 0-cycle $z'=z'_{n,S}$ of degree $0$ such that it has the same image as $\delta z_0-z_v$ in $A_0(X_v)/n$
for any $v\in S.$ We set $z=\delta z_0-z',$ it is a global 0-cycle of degree $\delta$ having
the same image as $z_v$ in $CH_0(X_v)/n$ for any $v\in S.$
\end{proof}

\begin{rem}\label{remark E to ii}
It is pointed out by Wittenberg that the exactness of $(E)$ implies
the assertion that the Brauer-Manin obstruction is the only obstruction
to the Hasse principle for 0-cycles of degree $1,$ \cite[Remark 1.1(iii)]{Wittenberg},
hence the implication (\textsc{Exact})$\Rightarrow$(0cyc-HP) is clear.
He also remarked that the exactness of $(E)$ implies automatically the exactness
of $(E_0)$ since $A_0(X_v)^\chapeau=A'_0(X_v):=A_0(X_v)/N_{\bar{k}_v|k_v}A_0(\overline{X}_v)$
if $v$ is an archimedean place, \cite[Remark 1.1(ii)]{Wittenberg}.
If we consider only the case $\delta=1,$
and want to prove the assertion that the Brauer-Manin obstruction is the only obstruction
to weak approximation for 0-cycles of degree $1,$ we shall suppose that there exists a family of
degree $1$ local 0-cycles $\{z_v\}\bot Br(X),$ therefore there exists a global 0-cycle of degree $1$
once we admit the exactness of $(E),$ we get unconditionally the implication
$\mbox{(\textsc{Exact})}\Rightarrow\mbox{(0cyc-WA)}.$

Even without existence of a global 0-cycle of degree $1,$ we can also argue directly with
$CH_0(-)$ and check that the exactness
of $(E)$ implies the assertion that the Brauer-Manin obstruction is the only obstruction
to weak approximation \emph{at finite places} for 0-cycles of degree $\delta$ on $X.$
In this case, we cannot reduce the question to $A_0(-)$ and the Brauer-Manin pairing only gives
information on $CH'_0(X_v)^\chapeau$ if $v$ is an archimedean place, where the latter group in general is
not isomorphic to $CH_0(X_v)^\chapeau,$ hence we cannot get approximation properties at the archimedean places.
\end{rem}

\section{Brauer-Manin obstruction, from rational points to 0-cycles}\label{BM obstruction}
We assume that $X$ is a variety such that the N\'eron-Severi group $NS(\overline{X})$ is torsion-free and
$H^1(\overline{X},\mathcal{O}_{\overline{X}})=H^2(\overline{X},\mathcal{O}_{\overline{X}})=0$
from \S \ref{comparison} to \S \ref{ii to ii'}. And we suppose further that $X$ is rationally
connected in \S \ref{E for RC}. Under these assumptions,
we prove Theorem B as well as the second assertion in Theorem A, \emph{i.e.}
for such a variety $X$
$\mbox{(pt-HP)}\Rightarrow\mbox{(0cyc-HP)}$ and
$\mbox{(pt-WA)}\Rightarrow\mbox{(0cyc-WA)}\Rightarrow\mbox{(\textsc{Exact})}.$
The proof divides into several steps; first of all, we prove
$\mbox{(pt-HP)}\Rightarrow\mbox{(0cyc-HP)}$ and $\mbox{(pt-WA)}\Rightarrow\mbox{(0cyc-WA)}$
in \S \ref{i to ii}; secondly, we prove $\mbox{(0cyc-WA)}\Rightarrow\mbox{(0cyc-WA)}^\delta$ in \S \ref{ii to ii'};
finally we complete the proof by proving $\mbox{(0cyc-WA)}^\delta\Rightarrow(\mbox{\textsc{Exact}})$
in \S \ref{E for RC}.

\subsection{Comparison of Brauer groups}\label{comparison}

We compare Brauer groups in this subsection, the following proposition
will be used in the proof in \S \ref{i to ii} and \S \ref{ii to ii'}.
It is a particular case of \cite[Theorem 2.3.1]{Harari2} by
Harari, we include a complete proof here for the convenience of the reader.

\begin{prop}\label{Brauer group}
Let $X$ be a proper smooth variety defined over a number field $k$ such that
the N\'eron-Severi group $NS(\overline{X})$ is torsion-free and
$H^1(\overline{X},\mathcal{O}_{\overline{X}})=H^2(\overline{X},\mathcal{O}_{\overline{X}})=0.$
Then there exists a finite extension $k'$ of $k$ satisfying the following condition:
for all finite extensions $l$ of $k$ linearly disjoint from $k'$ over $k,$
the homomorphism induced by restriction
$$Br(X)/Br(k)\to Br(X_l)/Br(l)$$
is an isomorphism of finite groups.
\end{prop}

\begin{proof}

The Hochschild-Serre spectral sequence
$$E_2^{p,q}=H^p(k,H^q(\overline{X},\mathbb{G}_m))\Rightarrow H^{p+q}(X,\mathbb{G}_m)$$
gives us an exact sequence
$$0\to H^1(k,Pic(\overline{X}))\to Br(X)/Br(k)\to Br(\overline{X})^{\Gamma_{k}}\to H^2(k,Pic(\overline{X})),$$
\emph{cf.} \cite[Proposition 2.2.1]{Harari2} for details.

Under the assumptions on $X,$ the Brauer group $Br(\overline{X})$ is finite,
and the Picard group $Pic(\overline{X})$ is torsion-free and finitely generated.
Let $k'$ be a sufficiently large finite Galois extension of $k$ such that $X(k')\neq\emptyset$ and
such that $\Gamma_{k'}$ acts trivially on $Pic(\overline{X})$ and on  $Br(\overline{X}).$
By construction, the Picard group $Pic(\overline{X})$ identifies with $Pic(X_{k'}).$
As $H^1(k',Pic(\overline{X}))=0,$ the group
$H^1(k,Pic(\overline{X}))$ identifies with $H^1(k'|k,Pic(X_{k'}))$
and we have an injection $H^2(k'|k,Pic(X_{k'}))\hookrightarrow H^2(k,Pic(\overline{X})),$
\emph{cf.} \cite[Propositions 4, 5 Ch.VII]{SerreCorpsLoc}.
As $Br(\overline{X})$ is finite, by enlarging $k'$ if necessary,
we may also assume that the image of $Br(\overline{X})^{\Gamma_{k}}$
in $H^2(k,Pic(\overline{X}))$ is contained in $H^2(k'|k,Pic(X_{k'})).$ The previous
exact sequence becomes
$$0\to H^1(k'|k,Pic(X_{k'}))\to Br(X)/Br(k)\to Br(\overline{X})^{\Gamma_{k'|k}}\to H^2(k'|k,Pic(X_{k'})).$$

Let $l\subset\bar{k}$ be an arbitrary finite extension of $k$ linearly disjoint from $k',$ we denote by $l'$ the composite of
$l$ and $k'.$ We obtain a commutative diagram by functoriality for $l/k:$
\SelectTips{eu}{12}$$\xymatrix@C=12pt @R=14pt{
0\ar[r] & H^1(k'|k,Pic(X_{k'}))\ar[r]\ar[d]& Br(X)/Br(k)\ar[r]\ar[d]& Br(\overline{X})^{\Gamma_{k'|k}}\ar[r]\ar[d]& H^2(k'|k,Pic(X_{k'}))\ar[d]\\
0\ar[r] & H^1(l'|l,Pic(X_{l'}))\ar[r]& Br(X_l)/Br(l)\ar[r]& Br(\overline{X}_l)^{\Gamma_{l'|l}}\ar[r]& H^2(l'|l,Pic(X_{l'})).
}$$
The first and the last two vertical homomorphisms are isomorphisms since
$l'/l$ is a Galois extension of group $\Gamma_{l'|l}\buildrel\simeq\over\to\Gamma_{k'|k}$
and $Pic(X_{k'})=Pic(X_{l'})=Pic(\overline{X}).$
Hence we obtain the desired isomorphism $$Br(X)/Br(k)\buildrel\simeq\over\to Br(X_l)/Br(l)$$
for all finite extensions $l/k$ linearly disjoint from $k'/k.$ The finiteness follows from the exact sequence.
\end{proof}

\subsection{Proof of \textmd{$\mbox{(pt-HP)}\Rightarrow\mbox{(0cyc-HP)}$} and \textmd{$\mbox{(pt-WA)}\Rightarrow\mbox{(0cyc-WA)}$}}
\label{i to ii}

Passing from the rational points to 0-cycles of degree $1,$ we prove the following theorem.
In order to deal with 0-cycles, we have to consider closed points, therefore we really need information concerning
$K$-rational points for all finite extensions $K$ of $k$ instead of information concerning only $k$-rational points.
That is why we state the main results in such a way instead of staying on the base field.

\begin{thm}\label{prop i-ii}
Let $X$ be a proper smooth variety defined over a number field $k$ such that
$NS(\overline{X})$ is torsion-free and
$H^1(\overline{X},\mathcal{O}_{\overline{X}})=H^2(\overline{X},\mathcal{O}_{\overline{X}})=0.$
For any finite extension $K/k,$ we assume that the Brauer-Manin obstruction is the only obstruction
to the Hasse principle (respectively, weak approximation) for $K$-rational points on $X_K.$

Then the Brauer-Manin obstruction is the only obstruction
to the Hasse principle (respectively, weak approximation) for 0-cycles of degree $1$ on $X_k.$
\end{thm}

\begin{proof}
Consider the projection $pr:X\times\mathbb{P}^1\to X,$ it has a section
$s:x\mapsto(x,0)$ where $0\in\mathbb{P}^1(k)$ is a chosen rational point.
Suppose that we can prove the statement
\begin{enumerate}
\item[$(\star)$ ]
the Brauer-Manin obstruction is the only obstruction
to the Hasse principle/weak approximation for 0-cycles of degree $1$ on $X\times\mathbb{P}^1,$
\end{enumerate}
then it is also the case for $X.$ In fact, let $\{z_v\}_v\bot Br(X)$ be a family of local 0-cycles of degree $1$
on $X,$ as $pr^*:Br(X)\to Br(X\times\mathbb{P}^1)$ is an isomorphism,
the family $\{s_*(z_v)\}_v$ on $X\times\mathbb{P}^1$ is orthogonal to $Br(X\times\mathbb{P}^1),$
where $s_*:CH_0(X)\to CH_0(X\times\mathbb{P}^1).$ The statement $(\star)$ gives a global 0-cycle of degree $1$
on $X\times\mathbb{P}^1,$ its image under $pr_*:CH_0(X\times\mathbb{P}^1)\to CH_0(X)$ gives the desired solution.

Under the same assumptions in the statement of the theorem, we are going to prove
$(\star).$ We set $\mathbb{P}^1_v=\mathbb{P}^1_{k_v}.$
Let $\{z_v\}_v\in \prod_{v\in\Omega_k}Z_0(X_v\times\mathbb{P}^1_v)$ be an arbitrary family of local 0-cycles of degree $1$
orthogonal to $Br(X\times\mathbb{P}^1),$ we fix a positive integer $n$ and a finite set $S$ of places of $k,$
we are going to approximate $z_v$ for $v\in S$ at the level $CH_0(-)/n.$
For the Hasse principle, the same argument applies.

We see from the exact sequence in the proof of Proposition \ref{Brauer group} that
$Br(X\times\mathbb{P}^1)$ is finite modulo $Br(k).$ Let $\{b_i\}_{1\leqslant i\leqslant s}$
be a complete set of representatives; we can choose a positive integer $m$
annihilating all $b_i\in Br(X\times\mathbb{P}^1)$ since the last group is torsion.
Let $k'$ be a finite extension of $k$ satisfying the condition of Proposition \ref{Brauer group},
we may also assume that $m$ is a multiple of $[k':k].$
We fix a global closed point $P$ of $X\times\mathbb{P}^1$ and denote its degree by
$\delta_P=[k(P):k].$

We denote by $\pi:X\times\mathbb{P}^1\to\mathbb{P}^1$ the projection. We fix integral models
$\mathcal{X}$ and $\mathbb{P}^1_{O_{k,S_0}}$ of $X$ and of $\mathbb{P}^1$ (proper and smooth) over $Spec(O_{k,S_0})$
where $S_0\supseteq S\cap\Omega_k^\infty$ is a certain finite set of places of $k.$
By enlarging $S_0$ if necessary, we may assume that all the $b_i$'s come from elements in
$Br(\mathcal{X}\times\mathbb{P}^1_{O_{k,S_0}}).$
For every $v\in\Omega_k\setminus S_0,$ as $Br(O_v)=0,$ the evaluation $\langle z_v,b_i\rangle_{k_v}$ is $0$
for any $i\in\{1,2,\ldots,s\}$ and for any 0-cycle on $X_v\times\mathbb{P}^1_v.$
Therefore we have
$$\sum_{v\in S_0}inv_v(\langle z_v,b_i\rangle_{k_v})=0$$ for every $b_i.$
Moreover, by enlarging $S_0$ if necessary,
we may also assume that for all $v\in\Omega_k\setminus S_0$ the reduction of $\mathcal{X}$
modulo $v$ has a $k(v)$-rational point by Lang-Weil estimation, hence $X_v$ has a $k_v$-point
by Hensel's lemma.

For each $v\in S_0,$ we write $z_v=z_v^+-z_v^-$ where $z_v^+$ and $z_v^-$ are effective 0-cycles
with disjoint supports. We pose $z_v^1=z_v+mn\delta_Pz_v^-=z_v^++(mn\delta_P-1)z_v^-,$ then
$deg(z_v^1)\equiv1(mod\mbox{ }mn\delta_P)$ and they are all effective 0-cycles.
We add to each $z_v^1$ a suitable multiple of the 0-cycle $mnP_v$ where $P_v=P\times_kk_v$
and we obtain $z_v^2$ of the same degree $\Delta\equiv1(mod\mbox{ }mn\delta_P)$ for all $v\in S_0.$
Then $\langle z_v,b_i\rangle_{k_v}=\langle z_v^1,b_i\rangle_{k_v}=\langle z_v^2,b_i\rangle_{k_v}$
since $m$ annihilates all the $b_i$'s, and the $z_v^2$'s are effective 0-cycles of degree $\Delta$
for all $v\in S_0.$ By Lemma \ref{moving lemma}, there exists for each $v\in S_0$ an effective
0-cycle $z_v^3$ close to $z_v^2$ with respect to the $b_i$'s and to $n$ such that
$\pi_*(z_v^3)$ is separable. We still have $\langle z_v,b_i\rangle_{k_v}=\langle z_v^3,b_i\rangle_{k_v}$
by continuity of the pairing. By using the last remark in \S \ref{subsec-0-cyc},
we also check that $z_v,z_v^1,z_v^2$ and $z_v^3$ have the same image
in $CH_0(X_v)/n.$

We choose a rational point $\infty\in\mathbb{P}^1(k)$ outside the supports of $\pi_*(z_v^3)$ for $v\in S_0.$
Then for each $v\in S_0$ we can write $\pi_*(z_v^3)-\Delta\infty=div(f_v)$ with $f_v\in k_v(\mathbb{P}^1_v)^*/k_v^*.$
Actually, each $f_v$ is a separable polynomial (can supposed to be monic) of degree $\Delta$ since $\pi_*(z_v)$
is effective separable and of degree $\Delta.$ We choose a place $v_0\in\Omega_k\setminus S_0,$
and choose $f_{v_0}$ a monic irreducible polynomial of degree $\Delta$ with coefficients in $k_{v_0},$
for example, an Eisenstein polynomial.
According to the weak approximation property for number fields, we obtain a monic polynomial $f$ of degree $\Delta$
with coefficients in $k$ such that it is sufficiently close to $f_v$ for all $v\in S_0\cup\{v_0\}.$

The polynomial $f$ is irreducible over $k_{v_0}$ by Krasner's lemma, a fortiori irreducible over $k,$ hence defines
a closed point $\theta$ of $\mathbb{P}^1$ of degree $\Delta.$ It is sufficiently close to
$\pi_*(z_v^3)$ for all $v\in S_0.$ More precisely, we write
$\theta_v=\theta\times_{\mathbb{P}^1}{\mathbb{P}_v^1}=\bigsqcup_{w\mid v,w\in\Omega_{k(\theta)}}Spec(k(\theta)_w)$
for $v\in\Omega_k,$ the image of $\theta$ in $Z_0(\mathbb{P}^1_v)$ is written as
$\theta_v=\sum_{w\mid v,w\in\Omega_{k(\theta)}}P_w$ where $P_w=Spec(k(\theta)_w)$ is a closed
point of ${\mathbb{P}^1_v}$ of residual field $k(\theta)_w.$
For each $v\in S_0,$ the 0-cycle $\theta_v$ is sufficiently close to $\pi_*(z^3_v),$ where
the effective separable 0-cycle $\pi_*(z^3_v)$ is written as $\sum_{w\mid v,w\in\Omega_{k(\theta)}}Q_w$ with
distinct $Q_w$'s.
Then by definition $k(\theta)_w=k_v(P_w)=k_v(Q_w),$ and $P_w$ is sufficiently close to $Q_w\in \mathbb{P}^1_v(k(\theta)_w).$
And we know that
$z^3_v$ is written as $\sum_{w\mid v,w\in\Omega_{k(\theta)}}M^0_w$ with $k_v(M^0_w)=k(\theta)_w$ and
$M^0_w\in X_v(k(\theta)_w)$ is situated
on the fiber of $\pi$ at the closed point $Q_w.$
The implicit function theorem implies that there exists a smooth $k(\theta)_w$-point $M_w$ on the fiber
$\pi^{-1}(\theta)$ close with respect to the $b_i$'s and to $n$
to $M^0_w$ for every $w\in S_0\otimes_kk(\theta).$ The closed point $M_w$ and $M_w^0$ have the same
image in $CH_0(X_v)/n,$ \emph{cf.} the last remark of \S \ref{subsec-0-cyc}.

We have for every $b_i\in Br(X\times\mathbb{P}^1)$ chosen previously, the equality on $X\times\mathbb{P}^1$
\begin{eqnarray*}
 \sum_{w\in S_0\otimes_kk(\theta)}inv_w( \langle M_w,b_i \rangle_{k(\theta)_w})\mbox{ }\mbox{ }\mbox{ }\mbox{ }\mbox{ }\mbox{ }\mbox{ }\mbox{ }&=&\sum_{w\in S_0\otimes_kk(\theta)}inv_w(b_i(M_w)) \\
=\sum_{w\in S_0\otimes_kk(\theta)}inv_w(b_i(M^0_w))\mbox{ }\mbox{ }\mbox{ }\mbox{ }\mbox{ }\mbox{ }\mbox{ }\mbox{ }\mbox{ }\mbox{ }\mbox{ }\mbox{ }\mbox{ }\mbox{ }\mbox{ }\mbox{ }& =&\sum_{v\in S_0}\sum_{w\mid v}inv_w(b_i(M^0_w))\\
=\sum_{v\in S_0}\sum_{w\mid v}inv_v(cores_{k(\theta)_w/k_v}(b_i(M^0_w)))& =&\sum_{v\in S_0}\sum_{w\mid v}inv_v( \langle M^0_w,b_i \rangle_{k_v})\\
=\sum_{v\in S_0}inv_v( \langle z^3_v,b_i \rangle_{k_v})\mbox{ }\mbox{ }\mbox{ }\mbox{ }\mbox{ }\mbox{ }\mbox{ }\mbox{ }\mbox{ }\mbox{ }\mbox{ }\mbox{ }\mbox{ }\mbox{ }\mbox{ }\mbox{ }\mbox{ }\mbox{ }\mbox{ }\mbox{ }\mbox{ }\mbox{ }& =&0
\end{eqnarray*}
We fix $M_w$ a $k(\theta)_w$-point on $\pi^{-1}(\theta)\simeq X_{k(\theta)}$ for each
$w\in\Omega_{k(\theta)}\setminus S_0\otimes_kk(\theta)$ (existence assured by the choice of $S_0$).
As indicated previously, we have $inv_w( \langle M_w,b \rangle_{k(\theta)_w})=0$
for all $b\in Br(X\times\mathbb{P}^1)$ and for all $w\notin S_0\otimes_kk(\theta)$
thanks to the integral model. Therefore,
$$\sum_{w\in \Omega_{k(\theta)}}inv_w( \langle M_w,b_i \rangle_{k(\theta)_w})=0.$$

By abuse of notation, here $b_i$ denotes also its image under the restriction
$i^*_\theta:Br(X\times\mathbb{P}^1)\to Br(\pi^{-1}(\theta))$ where $i_\theta:\pi^{-1}(\theta)\to X\times\mathbb{P}^1$
is the natural closed immersion. The above equality is viewed as a pairing on $\pi^{-1}(\theta)$
by functoriality of the Brauer-Manin pairing.

We consider an arbitrary element of $Br(X),$ it extends isomorphically to an element
of $Br(X\times\mathbb{P}^1)$ via the trivial fibration $pr:X\times\mathbb{P}^1\to X,$
then we restrict it to the fiber $\pi^{-1}(\theta)$ via $i^*_\theta$ and get an element of $Br(\pi^{-1}(\theta)),$
this gives simply its image under the restriction $Br(X)\to Br(X_{k(\theta)})$ if we identify
$\pi^{-1}(\theta)$ with $X_{k(\theta)}.$
By Proposition \ref{Brauer group}, if $k(\theta)$ and $k'$ are linearly disjoint over $k$
the $b_i$'s generate $Br(X_{k(\theta)})$ modulo $Br(k(\theta)).$
It is the case since $[k(\theta):k]=\Delta\equiv1(mod\mbox{ }[k':k])$ by construction.
Therefore $\{M_w\}_{w\in\Omega_{k(\theta)}}\bot Br(\pi^{-1}(\theta))$ since $Br(k(\theta))$ gives no
contribution to the Brauer-Manin pairing on the $k(\theta)$-variety $\pi^{-1}(\theta).$

By hypothesis, there exists a global $k(\theta)$-point on $\pi^{-1}(\theta)$ sufficiently close to $M_w$
for all $w\in S\otimes_kk(\theta),$ such that $M$ and $M_w$ have the same image in
$CH_0(\pi^{-1}(\theta)_w)/n.$ We check by construction that $M$ and $z_v$ have the same image in
$CH_0(X_v)/n$ for all $v\in S.$
The point $M$ viewed as a 0-cycle on $X$ is of degree $\Delta\equiv1(mod\mbox{ }n\delta_P),$
a suitable linear combination with the 0-cycle $nP$ gives a global 0-cycle
of degree $1$ having the image as $z_v$ in $CH_0(X_v)/n$ for all $v\in S.$
This completes the proof.
\end{proof}

\begin{rem}
The statement that we have proved is a particular case of a much more general result of the author,
Theorem 2.4 \cite{Liang2} applied to the trivial fibration
$\pi:X\times\mathbb{P}^1\to\mathbb{P}^1.$ The proof presented here is much simpler than the general case.
The same argument without taking care of the Brauer-Manin pairing and without comparison of
Brauer groups proves the following statement.
\end{rem}

\begin{prop}[personal discussion with Wittenberg]
Let $X$ be a proper smooth variety defined over a number field $k.$
For any finite extension $K/k,$ we assume that $X_K$ satisfies the Hasse principle (respectively, weak approximation)
for $K$-rational points.

Then $X$ satisfies the Hasse principle (respectively, weak approximation)
for 0-cycles of degree $1.$
\end{prop}

\subsection{Generalized Hilbert's irreducibility theorem}\label{Hilbert}
In this subsection,
we introduce the notion of generalized Hilbertian subsets and state an approximation property for these subsets
of $\mathbb{P}^1.$
This property is the crucial point for the proof in \S \ref{ii to ii'}.

\begin{defn}
Let $X$ be a geometrically integral variety defined over a number field $k.$
A subset $\textsf{Hil}\subset X$ of closed points of $X$ is called a
\emph{generalized Hilbertian subset} if there exists a finite \'etale morphism
$Z\buildrel\rho\over\to U\subset X$ where $U$ is a non-empty open set of $X$
and $Z$ is an integral variety (not necessarily geometrically integral over $k$)
such that $\textsf{Hil}$ equals the set of closed points $\theta$ of $U$
having connected fiber $\rho^{-1}(\theta).$
\end{defn}

\begin{rem}
If $\textsf{Hil}$ is a generalized Hilbertian subset of $X,$ then
$\textsf{Hil}\cap X(k)$ is a (classic) Hilbertian subset of $X,$ \emph{cf.} \cite[3.2]{Harari} for definition.
A classic Hilbertian subset can be empty, but a generalized Hilbertian subset is always non-empty,
\emph{cf.} \cite[\S 1]{Liang2}.
\end{rem}

\begin{prop}\label{Hilbertsubset}
Let $k$ be a number field and let $S$ be a finite set of places of $k.$
Let $\textsf{Hil}$ be a generalized Hilbertian subset of the $k$-projective line
$\mathbb{P}^1.$

Then for an arbitrary family of effective separable local 0-cycles $\{z_v\}_{v\in S}$ of degree $\Delta>0$ on $\mathbb{P}^1_v,$
there exists a global closed point $\theta$ of $\mathbb{P}^1$ of degree $\Delta$ such that
$\theta\in\textsf{Hil}$ and such that the image of $\theta$ in $Z_0(\mathbb{P}^1_v)$ is sufficiently close to $z_v$ for all $v\in S.$
\end{prop}

The property above was stated and proved by the author in a more general way for curves of arbitrary genus
in \cite{Liang1}. It was used to discuss properties of 0-cycles on fibrations over a curve of arbitrary genus,
the simpler statement here concerns only $\mathbb{P}^1,$ one can find a complete proof in \cite[Lemma 3.4]{Liang1}.

In some sense, this is a variant of Hilbert's irreducibility theorem,
we can approximate effective local 0-cycles by a global closed point in the generalized Hilbertian subset
$\textsf{Hil}:$ it is a generalized version for 0-cycles of Theorem 1.3 of Ekedahl \cite{Ekedahl}.

\subsection{Proof of \textmd{$\mbox{(0cyc-WA)}\Rightarrow\mbox{(0cyc-WA)}^\delta$}}\label{ii to ii'}

Since the sequence $(E)$ concerns 0-cycles of all degrees, it is natural to pass from 0-cycles of degree $1$
to 0-cycles of degree $\delta$ for arbitrary $\delta.$ Concerning this, we prove the following proposition.
This step plays a crucial role in the whole proof.
Comparing with \S \ref{i to ii}, the difference is the utilization of
the generalized Hilbert's irreducibility theorem explained in \S \ref{Hilbert}, which is the key point of this step.

\begin{prop}\label{prop ii-ii'}
Let $X$ be a proper smooth variety defined over a number field $k$ such that
the N\'eron-Severi group $NS(\overline{X})$ is torsion-free and
$H^1(\overline{X},\mathcal{O}_{\overline{X}})=H^2(\overline{X},\mathcal{O}_{\overline{X}})=0.$
For any finite extension $K/k,$ we assume that the Brauer-Manin obstruction is the only obstruction
to weak approximation for 0-cycles of degree $1$ on $X_K.$

Then for all integers $\delta$ the Brauer-Manin obstruction is the only obstruction
to weak approximation for 0-cycles of degree $\delta$ on $X_k.$
\end{prop}

\begin{proof}
First of all, we remark that we may assume (pt-HP)/(pt-WA) instead of (0cyc-HP)/(0cyc-WA) stated here;
no change is needed in the following proof.

We follow the proof of Theorem \ref{prop i-ii}, the degree of
$k(\theta)$ over $k$ is $\Delta$ with
$\Delta\equiv\delta(mod\mbox{ }[k':k]).$
If $\delta\neq1$ we cannot obtain directly by degree argument the linear disjointness
of $k(\theta)$ and $k'$ over $k.$ However, in order to be able to apply Proposition \ref{Brauer group}
we need the notion of generalized Hilbertian subset to obtain the linear disjointness.
In fact, let us consider the
finite \'etale morphism $\mathbb{P}^1_{k'}\to \mathbb{P}^1,$ it defines a generalized Hilbertian subset
$\textsf{Hil}$ of $\mathbb{P}^1.$ A closed point $\theta$ of $\mathbb{P}^1$ belongs to $\textsf{Hil}$
if and only if $\theta\times_{\mathbb{P}^1}\mathbb{P}^1_{k'}\simeq Spec(k(\theta)\otimes_kk')$
is connected, \emph{i.e.} the fields $k(\theta)$ and $k'$ are linearly disjoint over $k.$
To complete the proof, it suffices
to find a closed point $\theta$ of $\mathbb{P}^1$ such that $\theta$ is contained in $\textsf{Hil}$
and such that $\theta$ is sufficiently close to $\pi_*(z_v^3)$ for all $v\in S_0.$
For this purpose, we replace
the argument using polynomials in the proof of Theorem \ref{prop i-ii} by
Proposition \ref{Hilbertsubset}, which applies since the 0-cycles $\pi_*(z_v^3)$
are effective separable and of positive degree $\Delta.$
\end{proof}

\begin{rem}
The additional $\mathbb{P}^1$ that we introduced in the proof of \ref{prop i-ii} and \ref{prop ii-ii'}
does not affect much the Chow groups and the Brauer groups, but it has two advantages.
First, it gives us an extra $1$-dimensional space to move the 0-cycles; even if the 0-cycles cannot move freely along
the direction of $X,$ we can still move them along the direction of $\mathbb{P}^1.$ Second, it gives us
a structure of fibration $X\times\mathbb{P}^1\to\mathbb{P}^1$ and fibration methods can be applied.
That is exactly what we have done in the proof.
\end{rem}

\subsection{Exactness of $(E)$ for rationally connected varieties}\label{E for RC}
In this subsection, we suppose further that $X$ is rationally connected.
We prove the proposition stated below, with which we arrive at the implication
$\mbox{(0cyc-WA)}^\delta\Rightarrow\mbox{(\textsc{Exact})}$ for rationally connected varieties.
It is a statement without base field extensions.
The proof is essentially an argument of Wittenberg \cite[Proposition 3.1]{Wittenberg}.
In the case where $X$ is a variety fibered over $\mathbb{P}^1$ with rationally connected generic fiber, the
following proposition is just a special case of Proposition 3.1 of \cite{Wittenberg}.
We rewrite a simplified proof here for the convenience of the reader.

The sequence $(E)$ concerns all the places $v\in\Omega_k$ while the weak approximation property concerns only
finitely many places $v\in S,$ the theorem of Koll\'ar/Szab\'o on Chow groups of 0-cycles of rationally connected varieties
fills the gap.

\begin{prop}\label{ii' to E}
Let $X$ be a proper, smooth, rationally connected variety defined over a number field $k.$
We assume that for all integers $\delta$ the Brauer-Manin obstruction is the only obstruction to
weak approximation for 0-cycles of degree $\delta.$

Then the sequence $(E)$ is exact for $X.$
\end{prop}

In order to prove the proposition, we need several lemmas.

\begin{lem}[\cite{Wittenberg}, Lemmas 1.11, 1.12]\label{finite exp}
Let $f:A\to B$ be a homomorphism of abelian groups, we denote by $\hat{f}:A^\chapeau\to B^\chapeau$
the inverse limit $\varprojlim_n$ of $f/n:A/nA\to B/nB.$

(1) If $Coker(f)$ is of finite exponent, then the induced homomorphism $Coker(f)\to Coker(\hat{f})$ is injective.

(2) If $Coker(f)$ and $Ker(f)$ are of finite exponent and if $Ker(B\buildrel{.n}\over\to B)$ is
finite for all integers $n>0,$ then the induced homomorphism $Ker(f)\to Ker(\hat{f})$ is surjective.
\end{lem}

\begin{lem}[\cite{CT05}, Proposition 11]\label{finite exp ker}
Let $X$ be a proper, smooth, rationally connected variety defined over a field $k$ of characteristic zero.

Then there exists an integer $n>0,$ such that for all extensions $k'$ of $k,$ the group
$A_0(X_{k'})$ is annihilated by $n.$
\end{lem}

For any proper smooth $k$-variety, let us denote by $deg_{k'}:CH_0(X_{k'})\to\mathbb{Z}$
the degree map over any field extension $k'$ of $k,$ and denote by $\widehat{deg}_{k'}:CH_0(X_{k'})^\chapeau\to\widehat{\mathbb{Z}}$
its completion.

\begin{lem}\label{finite exp coker}
Let $X$ be a proper, smooth, rationally connected variety defined over a field $k$ of characteristic zero.

Then there exists an integer $n>0,$ such that for all extensions $k'$ of $k,$
the groups $Coker(deg_{k'})$ and $Coker(\widehat{deg}_{k'})$ are annihilated by $n.$
\end{lem}

\begin{proof}
We denote by $n$ the index of the variety $X_{/k},$ in particular there exists a 0-cycle of degree $n$ on $X.$
We have therefore a surjection $\mathbb{Z}/n\mathbb{Z}\to Coker(deg_{k'})$ for all extensions $k'$ of $k.$
By tensoring with $\widehat{\mathbb{Z}},$ we get also a surjection
$\widehat{\mathbb{Z}}/n\widehat{\mathbb{Z}}\to Coker(\widehat{deg}_{k'}).$
Hence $Coker(deg_{k'})$ and $Coker(\widehat{deg}_{k'})$ are both annihilated by $n.$
\end{proof}

\begin{lem}[\cite{Kollar-Szabo}, Corollary 9]\label{KS thm}
Let $X$ be a proper, smooth, rationally connected variety defined over a number field $k.$

Then there exists a finite subset $S\subset\Omega_k$ such that for  all $v\in\Omega_k\setminus S$
we have an isomorphism $deg_v:CH_0(X_v)\buildrel\simeq\over\to\mathbb{Z}.$
\end{lem}

\begin{proof}
Since $CH_0(-)$ is a birational invariant for proper smooth varieties, we can suppose that $X$ is smooth and projective
by Hironaka's resolution of singularity.
For a sufficiently large finite subset $S\subset\Omega_k,$ there exists a smooth projective morphism
$\mathcal{X}\to Spec(O_{k,S})$ such that its generic fiber is isomorphic to $X$ and all the fibers are separably
rationally connected, \emph{cf.} \cite[Theorem 3.11]{Kollar}. We refer to \cite[IV.3.]{Kollar} for the notion
of separably rational connectedness, the appropriate notion for varieties defined over a field of positive characteristic.
Thanks to a theorem of Koll\'ar/Szab\'o \cite[Theorem 5]{Kollar-Szabo},
we get $Ker(deg_v:CH_0(X_v)\to\mathbb{Z})=0$ for all $v\in\Omega_k\setminus S.$

Moreover, according to the Lang-Weil estimation, by enlarging
$S$ if necessary, we may assume that for all finite places $v\in\Omega_k\setminus S$ the reduction of
$\mathcal{X}$ modulo $v$ has a $k(v)$-rational point. Therefore $X_v$ admit a $k_v$-rational point
by Hensel's lemma, and we get the surjectivity of $deg_v:CH_0(X_v)\to\mathbb{Z}$
for all $v\in\Omega_k\setminus S.$
\end{proof}

\begin{proof}[Proof of Proposition \ref{ii' to E}]
Let $\{z_v\}_{v\in\Omega_k}\bot Br(X)$ be a family of local 0-cycles of degree $\delta.$ According to
the hypothesis, for any integer $n>0$ and any finite subset $S\subset\Omega_k$ there exists
a global 0-cycle $z=z_{n,S}$ of degree $\delta$ such that $z$ and $z_v$ have the same image
in $CH_0(X_v)/2n.$ Hence we obtain the following property for all finite subsets $S\subset\Omega_k.$
\begin{enumerate}
\item[$(P_S)$]
Let $\{z_v\}_{v\in\Omega_k}\bot Br(X)$ be a family of local 0-cycles of degree $\delta,$
then for any integer $n>0$ there exists a global 0-cycle $z=z_{n,S}$ of degree $\delta$ on $X$
such that $z=z_v$ in $CH_0(X_v)/n$ if $v\in S\cap\Omega_k^\textmd{f}$ and
$z=z_v+N_{\bar{k}_v|k_v}(u_v)$ in $CH_0(X_v)$ for a certain (class of) 0-cycle $u_v\in CH_0(\overline{X}_v)$
if $v\in S\cap\Omega_k^\infty.$
\end{enumerate}

We explain how to deduce the exactness of $(E)$ for $X$ from the property $(P_S).$
We start from an element $\{\hat{z}'_v\}_{v\in\Omega_k}$ of $\prod_{v\in\Omega_k}CH'_0(X_v)^\chapeau$$\mbox{ }$
sent to $0$ in the group $Hom(Br(X),\mathbb{Q}/\mathbb{Z}),$
the following lemma \ref{step1} shows that it can be written as a sum of the image of a certain element of
$CH_0(X)^\chapeau$$\mbox{ }$ and the image of a family $\{z'_v\}_{v}\in\prod_{v\in\Omega_k}CH'_0(X_v).$
Moreover $\{z'_v\}_{v}$ can be lifted to a family $\{z_v\}_v\in\prod_{v\in\Omega_k}CH_0(X_v)$ such that all components
have the same degree $\delta.$
Once this is done, according to Lemma \ref{KS thm}, we choose a finite subset $S\subset\Omega_k$
containing all the archimedean places $\Omega_k^\infty$ such that $CH_0(X_v)\simeq\mathbb{Z}$ for all
$v\notin S,$ and we choose a positive integer $n$ annihilating $A_0(X_v)$ for all $v\in\Omega_k$
by Lemma \ref{finite exp ker}. The property $(P_S)$ applied to $\{z_v\}_v$
for the integer $n$ will give tautologically the exactness of $(E)$ for $X.$
\end{proof}

In order to complete the proof, it suffices to prove the following lemma.

\begin{lem}\label{step1}
Every family $\{\hat{z}'_v\}_{v}\in\prod_{v\in\Omega_k}CH'_0(X_v)^\chapeau$$\mbox{ }$ orthogonal to $Br(X)$
can be written as a sum of the image of a family $\{z'_v\}_{v}\in\prod_{v\in\Omega_k}CH'_0(X_v)$
and the image of a certain element of $CH_0(X)^\chapeau\mbox{ }.$

Moreover, the family $\{z'_v\}_{v}$ can be lifted to a family in $\prod_{v\in\Omega_k}CH_0(X_v)$ such that all components
have the same degree.
\end{lem}

\begin{proof}
First of all, we consider the commutative diagram induced by functoriality for
the structural morphism $X\to Spec(k):$
\SelectTips{eu}{12}$$\xymatrix@C=20pt @R=14pt{
CH_0(X)^\chapeau\ar[d]^{\widehat{deg}}\ar[r]& \prod_{v\in\Omega_k}CH'_0(X_v)^\chapeau \ar[d]^{\prod_v\widehat{deg'_v}}\ar[r] & Hom(Br(X),\mathbb{Q}/\mathbb{Z})\ar[d]\\
CH_0(k)^\chapeau\ar[r]& \prod_{v\in\Omega_k}CH'_0(k_v)^\chapeau \ar[r] & Hom(Br(k),\mathbb{Q}/\mathbb{Z}).
}$$
The second line is exact since it is identified as
$$\widehat{\mathbb{Z}}\to\prod_{v\in\Omega^\textmd{f}}\widehat{\mathbb{Z}}\times\prod_{v\in\Omega^\mathbb{R}}\mathbb{Z}/2\mathbb{Z}\to Hom(Br(k),\mathbb{Q}/\mathbb{Z}),$$
which is the Pontryagin dual of the exact sequence coming from the class field theory
$$0\to Br(k)\to\bigoplus_{v\in\Omega_k}Br(k_v) \to\mathbb{Q}/\mathbb{Z}\to 0.$$
There exists then $\hat{\delta}\in\widehat{\mathbb{Z}}$ having image $\{\widehat{deg'_v}(\hat{z}'_v)\}_v$
in $\prod_{v\in\Omega^\textmd{f}}\widehat{\mathbb{Z}}\times\prod_{v\in\Omega^\mathbb{R}}\mathbb{Z}/2\mathbb{Z}.$

The cokernel $Coker(\widehat{deg}:CH_0(X)^\chapeau\to\widehat{\mathbb{Z}})$
is of finite exponent by lemma \ref{finite exp coker}, at the same time
$\widehat{\mathbb{Z}}/\mathbb{Z}$ is divisible, hence the homomorphism
$\mathbb{Z}\times CH_0(X)^\chapeau\to\widehat{\mathbb{Z}},(\delta,\hat{z})\mapsto\delta+\widehat{deg}(\hat{z})$
is surjective.
Therefore, we write $\hat{\delta}=\delta+\widehat{deg}(\hat{z})$ with $\delta\in\mathbb{Z}$ and
$\hat{z}\in CH_0(X)^\chapeau\mbox{ },$
by replacing $\{\hat{z}'_v\}_v$ by $\{\hat{z}'_v-\hat{z}\}_v$ if necessary, we may assume
that $\{\widehat{deg}_v(\hat{z}'_v)\}_v=\delta$ in
$\prod_{v\in\Omega^\textmd{f}}\widehat{\mathbb{Z}}\times\prod_{v\in\Omega^\mathbb{R}}\mathbb{Z}/2\mathbb{Z}$
with $\delta\in\mathbb{Z}.$

In this paragraph, we prove that the element $\{\hat{z}'_v\}_v$ comes from an element $\{z'_v\}_v$ of
$\prod_vCH'_0(X_v).$ In fact, the cokernel
$$Coker\left(\prod_vdeg'_v:\prod_vCH'_0(X_v)\To \prod_vCH'_0(k_v)=\prod_{v\in\Omega^\textmd{f}}{\mathbb{Z}}\times\prod_{v\in\Omega^\mathbb{R}}\mathbb{Z}/2\mathbb{Z}\right)$$
is of finite exponent (Lemma \ref{finite exp coker}),
Lemma \ref{finite exp}(1) implies that the image of $\delta$ in
$\prod_vCH'_0(k_v)=\prod_{v\in\Omega^\textmd{f}}{\mathbb{Z}}\times\prod_{v\in\Omega^\mathbb{R}}\mathbb{Z}/2\mathbb{Z}$
is equal to $\{deg'_v(z'_{0,v})\}_v$ for a certain $\{z'_{0,v}\}_v\in\prod_vCH'_0(X_v).$
Therefore $\{\widehat{deg'_v}(z'_{0,v}-\hat{z}'_v)\}_v=\delta-\delta=0$ in
$\prod_vCH'_0(k_v)^\chapeau=\prod_{v\in\Omega^\textmd{f}}{\widehat{\mathbb{Z}}}\times\prod_{v\in\Omega^\mathbb{R}}\mathbb{Z}/2\mathbb{Z}.$
For each $v\in\Omega^\textmd{f}_k,$ the kernel and the cokernel of $deg_v:CH_0(X_v)\to\mathbb{Z}$
are of finite exponent (Lemmas \ref{finite exp ker}, \ref{finite exp coker}),
we apply lemma \ref{finite exp}(2), this implies that $z'_{0,v}-\hat{z}'_v\in CH'_0(X_v)^\chapeau=CH_0(X_v)^\chapeau\mbox{ }$
comes from an element of $CH'_0(X_v)=CH_0(X_v),$ hence so is $\hat{z}'_v$ and we denote one of its preimages
by $z'_v\in CH'_0(X_v).$ For each $v\in\Omega_k^\mathbb{R},$ knowing that
the kernel and the cokernel of $deg'_v$ are of exponent $2,$ the same argument as above works
and we denote one of the preimages of $\hat{z}'_v$ by $z'_v\in CH'_0(X_v).$
For each $v\in\Omega_k^\mathbb{C},$ we set $z'_v=0\in CH'_0(X_v)=0.$

In the rest of the proof, we look for a lifting of each $z'_v\in CH'_0(X_v)$ to a 0-cycle of the same degree $\delta.$
Notice that the homomorphism
$$\prod_vCH'_0(k_v)=\prod_{v\in\Omega^\textmd{f}}{\mathbb{Z}}\times\prod_{v\in\Omega^\mathbb{R}}\mathbb{Z}/2\mathbb{Z}\To \prod_vCH'_0(k_v)^\chapeau=\prod_{v\in\Omega^\textmd{f}}\widehat{\mathbb{Z}}\times\prod_{v\in\Omega^\mathbb{R}}\mathbb{Z}/2\mathbb{Z}$$
is injective, the images of $\delta$ and of $\{z'_v\}_v$ are therefore the same in
$\prod_{v\in\Omega^\textmd{f}}{\mathbb{Z}}\times\prod_{v\in\Omega^\mathbb{R}}\mathbb{Z}/2\mathbb{Z},$
hence for each $v\in\Omega_k^\infty$ we can choose
in $CH_0(X_v)$ a preimage of $z'_v\in CH'_0(X_v)$ such that
it is of degree $\delta.$
\end{proof}

\section{Some remarks on the main results}\label{remark}

We give some remarks on the concerning topic: arithmetic of rational points versus arithmetic
of 0-cycles.

\subsection{}

The image of $Br(k)$ in $Br(X)$ contributes nothing to the Brauer-Manin pairing
for 0-cycles of the same degree, the sequence $(E_0)$ can be
rewritten as
$$A_0(X)^\chapeau\to\prod_{v\in\Omega}A_0(X_v)^\chapeau\to Hom(Br(X)/Br(k),\mathbb{Q}/\mathbb{Z}).$$
The quotient $Br(X)/Br(k)$ is finite if $X$ is rationally connected, Proposition \ref{Brauer group}, the
cokernel of the global-local mapping above is finite once $(E_0)$ is exact.

\subsection{}
Besides certain homogeneous spaces, our main results also apply to Ch\^atelet surfaces
and reprove the exactness of $(E)$ for these surfaces thanks to the work of Colliot-Th\'el\`ene/Sansuc/Swinnerton-Dyer
\cite[Theorem 8.11]{chateletsurfaces}. This is already known by Frossard \cite{Frossard}
and van Hamel \cite{vanHamel}.

\subsection{}
Even though the converse statement is correct for all proper smooth varieties,
it is still unknown, \emph{without assumption for base field extensions}, whether the statement that
the Brauer-Manin obstruction is the only obstruction to weak approximation for 0-cycles of degree $1$
implies the exactness of $(E)$ for rationally connected varieties. The only gap is passing from 0-cycles of degree $1$
to arbitrary degree $\delta$ with a certain argument \emph{staying on the base field}.
In \S \ref{ii to ii'}, we only prove that $\mbox{(0cyc-WA)}$ implies $\mbox{(0cyc-WA)}^\delta,$
where we need an assumption with base field extensions since when applying the fibration method to
$X\times\mathbb{P}^1\to\mathbb{P}^1$ we need to consider closed points of $\mathbb{P}^1$ in our argument.
We believe that the answer to the following question is positive.

\begin{pb*}
For rationally connected varieties $X,$ does the statement:
\small
\begin{enumerate}
\item[]
the Brauer-Manin obstruction is the only obstruction to weak approximation for 0-cycles of degree $1$ on $X$
\end{enumerate}\normalsize
imply the exactness of $(E)$ for $X$?
\end{pb*}

\subsection{}
Considering the statement (pt-WA), one may ask:
Is there a rationally connected $k$-variety $X$ such that the Brauer-Manin obstruction is the only
obstruction to weak approximation for rational points but it is not the case for $X_K$
for a certain finite extension $K/k$?

Counter-examples to the statement that the Brauer-Manin obstruction is the only
obstruction to weak approximation for rational points
were given by Skorobogatov \cite{Sk-beyond} and by Poonen \cite{Poonen}, they are
not rationally connected.
Our expected answer to this question is no, but it may have positive answer for more general varieties.
The following is proposed by Harari: We start from a $k$-bielliptic surface $X=(C_1\times C_2)/G,$
where $C_1$ and $C_2$ are smooth curves of genus $1$ and $G$ acts freely on $C_1\times C_2,$
such that $X(k_v)=\emptyset$ for a certain place $v$ of $k.$ It is trivially true that
the Brauer-Manin obstruction
is the only obstruction to weak approximation for $k$-rational points. We take a finite extension $K/k$
such that $C_1$ and $C_2$ are of positive rank over $K,$ then
$\overline{X(K)}\subsetneq \left[\prod_{w\in\Omega_K}X(K_w)\right]^{Br(X_K)},$
\cite[Corollary 6.3]{HarariENS}, the Brauer-Manin obstruction is not the only one to weak approximation
on $X_K.$

However, the following similar question seems easier to answer even in the family of rational varieties:

Is there a $k$-variety $X$ satisfying weak approximation for rational points but
it is not the case for $X_K$ for a certain finite extension $K/k$?

We consider the Ch\^atelet surface over a number field $k$ defined by the equation
$x^2-ay^2=P(z)$ where $P\in k[z]$ is a polynomial of degree $4$ and $a\in k^*\setminus k^{*2}.$
Assume that $P$ is irreducible over $k,$ then $Br(X)/Br(k)$ is trivial and hence weak approximation for $k$-rational points
holds for $X,$ \cite[Theorem 8.11]{chateletsurfaces}. The polynomial $P$ will become reducible
over a certain finite extension $K/k,$ one may have that $Br(X_K)/Br(K)$ is non-trivial, this may probably give
an obstruction
$\left[\prod_{w\in\Omega_K}X(K_w)\right]^{Br(X_K)}\subsetneq \prod_{w\in\Omega_K}X(K_w)$
to weak approximation on $X_K.$ Hopefully we can find some specific examples
by explicit calculations.

\subsection{}
Shall we expect the converse implications
$\mbox{(0cyc-HP)}\Rightarrow\mbox{(pt-HP)},$ $\mbox{(0cyc-WA)}\Rightarrow\mbox{(pt-WA)}$
and $\mbox{(\textsc{Exact})}\Rightarrow\mbox{(pt-HP,WA)}$?

The answer is negative if we do not restrict to a certain family of varieties.
Consider Poonen's $3$-fold \cite{Poonen}, it is a Ch\^atelet-surface bundle over an elliptic curve
of rank $0$ with finite Tate-Shafarevich group. The Brauer-Manin obstruction (even if applied to
\'etale coverings) is not the only obstruction to the Hasse principle for rational points, \cite{Poonen}.
But Colliot-Th\'el\`ene (respectively, the author) proved that,
it is the only obstruction to the Hasse principle (respectively, weak approximation)
for 0-cycles of degree $1,$ \cite[Theorem 3.1]{CTsurPoonen}, \cite[Theorems 2.1, 2.2]{Liang1};
the exactness of $(E)$ was also established by the author in \cite[Proposition 5.1]{Liang3}.
We also remark that exactness of $(E)$ was proved by Colliot-Th\'el\`ene for all curves
(assuming finiteness of Tate-Shafarevich groups), \cite[Proposition 3.3]{CT99HP0-cyc}, but
(pt-HP) is still a conjecture by Skorobogatov \cite{Skbook}.

Even for rationally connected varieties, the answer is not clear.
We consider a Severi-Brauer-variety bundle over $\mathbb{P}^n,$ it is a rational variety.
The author proved that the Brauer-Manin obstruction is the only obstruction
to the Hasse principle/weak approximation for 0-cycles of degree $1$ on such a variety, \cite[Corollary D]{Liang2}.
With Schinzel's hypothesis, Wittenberg proved the similar statement for rational points, \cite[Corollary 3.6]{WittenbergLNM}.
But it is still unknown without this additional hypothesis.

\subsection{}
Instead of considering the whole Brauer group, we may also consider obstructions
related to subgroups of the Brauer group or other types of obstructions. Here are
two usual considerations as follows.

\subsubsection{}
We may consider the algebraic Brauer-Manin obstruction related to the ``subgroup''
$Ker(Br(X)/Br(k)\to Br(\overline{X}))\simeq H^1(k,Pic(\overline{X}))$
of the whole Brauer group. Once the geometric Picard group $Pic(\overline{X})$ is torsion-free (hence finitely generated),
the comparison under base field extensions of Brauer groups holds, \emph{cf.} Proposition \ref{Brauer group}.
The same argument of Theorem \ref{prop i-ii} proves the following statement:
\begin{enumerate}
\item[] If the geometric Picard group $Pic(\overline{X})$ is torsion-free,
then the \emph{algebraic} Brauer-Manin obstruction is the only obstruction to the Hasse principle/weak approximation
for rational points on $X_K$ for all finite extensions $K/k$ implies that for
0-cycles of degree $1.$
\end{enumerate}
All K3 surfaces have torsion-free geometric Picard groups,
but unfortunately for general K3 surfaces the transcendental elements in
the Brauer group may really give an obstruction to weak approximation, \cite[Theorem 3.3]{Wittenberg04} and \cite{HVV}.
Up to now, the Brauer-Manin obstruction is not well understood for general K3 surfaces.

\subsubsection{}
We may also consider the subgroup of locally constant elements
$$\ba(X)=Ker\left(\frac{Br(X)}{Br(k)}\To\prod_{v\in\Omega_k}\frac{Br(X_v)}{Br(k_v)}\right),$$
it is a finite group once finiteness of Tate-Shafarevich groups of certain abelian varieties is supposed,
\cite[Proposition 2.14]{B-CT-Sk}. In several cases, the Brauer-Manin obstruction associated to this (conjecturally finite) part
is enough to explain the failure of the Hasse principle.

Eriksson/Scharaschkin prove that the $\ba(X)$-obstruction is the only obstruction to the
Hasse principle for 0-cycles of degree $1$ on any smooth curve $C$ assuming finiteness
of $\sha(Jac(C),k),$ \cite{Eriksson}. However, it is not the case for rational points.
In fact, consider a curve $C_{/k}$ without any $k$-rational points but with a global 0-cycle of degree $1,$
then there exists a family of local 0-cycles of degree $1$ orthogonal to $\ba(X).$
If moreover, the curve $C$ has $k_v$-points $x_v$ for all $v\in\Omega_k,$
then $\{x_v\}\bot \ba(X)$ since $\ba(X)$ consists of only the locally constant elements.
Therefore $\ba(X)$-obstruction is not the only obstruction to the Hasse principle for rational points.
For example, the equation $3x^4+4y^4=19z^4$ defines such a plane curve over $k=\mathbb{Q},$
\cite[p.128]{Skbook}.

In order to look for the relation between $\ba(X)$-obstruction for rational points
and for 0-cycles,
one may try to compare $\ba(X)$ under base field extension for a certain family of varieties,
but several difficulties come out. On the other hand, we have another approach.
Considering the Hasse principle, we should assume that $X$ has $k_v$-points for all places $v.$
Under this assumption, Wittenberg \cite[Theorem 3.3.1]{Wittenberg2008} proved that
the $\ba(X)$-obstruction vanishes for rational points (\emph{i.e.} there exists
a family of local rational points $\{x_v\}_{v\in\Omega_k}\bot\ba(X)$) if and only if the elementary
obstruction $ob(X)$ vanishes, assuming the finiteness of the Tate-Shafarevich group of the Picard
variety of $X$ over $k,$ \emph{cf.} \cite[\S 2]{Skbook} for the definition of the elementary obstruction.
The same proof shows that if there exists for each $v$ a 0-cycle of degree $1$ over $X_v,$ then
the $\ba(X)$-obstruction vanishes if and only if $ob(X)=0,$ assuming $\sha(\underline{Pic}^0_{X/k},k)<\infty.$

\begin{prop}
Let $X$ be a proper smooth variety defined over a number field $k.$ Suppose that
$\sha(\underline{Pic}^0_{X/k},K)<\infty$ for any finite extension $K/k.$

If $\ba(X_K)$-obstruction is the only obstruction to Hasse principle for $K$-rational points
on $X_K$ for any finite extension $K/k,$ then $\ba(X)$-obstruction is the only obstruction
to Hasse principle for 0-cycles of degree $1$ on $X.$
\end{prop}

\begin{proof}
With the same notation and the same argument as Theorem \ref{prop i-ii}, we start from a family
of local 0-cycles of degree $1$ with condition $\{z_v\}_v\bot\ba(X).$
From $\{s_*(z_v)\}\in\prod_{v\in\Omega_k}Z_0(X_v\times\mathbb{P}_v^1),$ under several procedures,
we get a family of local points $\{M_w\}_{w\in\Omega_{k(\theta)}}$ on the $k(\theta)$-variety
$\pi^{-1}(\theta).$ On the other hand, the existence of $\{z_v\}_v\bot\ba(X)$ implies that
$ob(X)=0,$ hence $ob(\pi^{-1}(\theta))=0$ by identifying $\pi^{-1}(\theta)$ with $X_{k(\theta)},$
\cite[Corollary 3.3.2]{Wittenberg2008}. Therefore $\{M_w\}_w\bot\ba(\pi^{-1}(\theta))$
according to the discussion before the proposition. The hypothesis gives us a $k(\theta)$-point
$M$ on $\pi^{-1}(\theta),$ which can be viewed as a global 0-cycle of degree $\Delta\equiv1(mod\mbox{ }\delta_P)$
on $X\times\mathbb{P}^1,$
we obtain a global 0-cycle of degree $1$ on $X\times\mathbb{P}^1$ and hence on $X.$
\end{proof}

\begin{cor}
Let $Y$ be a
homogeneous space of a connected linear algebraic group $G$ defined over a number field
and let $X$ be a smooth compactification of $Y.$ Suppose that the geometric stabilizer
of $Y$ is connected (or abelian if $G$ is supposed semisimple simply connected).

Then the $\ba(X)$-obstruction is the only obstruction to the Hasse principle for 0-cycles of degree $1$
on $X.$
\end{cor}

\begin{rem}
It is known to experts that on (smooth compactifications of) certain homogeneous spaces (including all those discussed in
the present work) the existence of a (global or local) 0-cycle of degree $1$ is equivalent to
the existence of a rational point. Several different arguments were given by Borovoi, Demarche, Starr,
and the author. The easiest one is due to Starr: the elementary obstruction is of abelian nature
(hence one can use the restriction-corestriction argument), with the help of \cite[Theorem 3.3(or 3.8)]{B-CT-Sk}
one can conclude for all homogeneous spaces considered in \cite{B-CT-Sk}.
With this observation, we can also arrive at the previous corollary directly.
However, this is not always the case for an arbitrary homogeneous space, see \cite{Florence}
for counter-examples.
\end{rem}

%
%
%
%


\bibliographystyle{alpha}
\bibliography{mybib1}
\end{document}